\documentclass[11pt]{amsart}
\usepackage{amssymb, latexsym, mathrsfs}

\numberwithin{equation}{section}
\newtheorem{Thm}[equation]{Theorem}
\newtheorem{Prop}[equation]{Proposition}
\newtheorem{Cor}[equation]{Corollary}
\newtheorem{Lem}[equation]{Lemma}

\theoremstyle{definition}

\newtheorem{Rmk}[equation]{Remark}

\begin{document}

\title [Rank $2$ hyperbolic Kac-Moody algebras and Hilbert modular forms]
{Rank $2$ symmetric hyperbolic Kac-Moody algebras and Hilbert modular forms}
\author[Henry Kim]{Henry H. Kim$^{\star}$}
\thanks{$^{\star}$ partially supported by an NSERC grant.}
\address{Department of
Mathematics, University of Toronto, Toronto, ON M5S 2E4, CANADA and Korea Institute for Advanced Study, Seoul, Korea}
\email{henrykim@math.toronto.edu}
\author[Kyu-Hwan Lee]{Kyu-Hwan Lee}
\address{Department of
Mathematics, University of Connecticut, Storrs, CT 06269, U.S.A. and Korea Institute for Advanced Study, Seoul, Korea}
\email{khlee@math.uconn.edu}
\begin{abstract}
In this paper we study rank two symmetric hyperbolic Kac-Moody algebras $\mathcal H(a)$ with the Cartan matrices $\begin{pmatrix} 2 & -a \\ -a & 2 \end{pmatrix}$, $a \ge 3$ and their automorphic correction in terms of Hilbert modular forms.
We associate a family of $\mathcal H(a)$'s to the quadratic field $\mathbb Q(\sqrt p)$ for each odd prime $p$ and show that there exists a chain of embeddings in each family.
When $p=5, 13, 17$, we show that the first $\mathcal H (a)$ in each family, i.e. $\mathcal H(3), \mathcal H(11), \mathcal H(66)$, is contained in a generalized Kac-Moody superalgebra whose denominator function is a Hilbert modular form given by a Borcherds product. Hence, our results provide automorphic correction for those $\mathcal H(a)$'s. We also compute asymptotic formulas for the root multiplicities of the generalized Kac-Moody superalgebras using the fact that the exponents in the Borcherds products are Fourier coefficients of weakly holomorphic modular forms of weight $0$.
\end{abstract}

\maketitle

\section*{Introduction}

Rank two symmetric Kac-Moody algebras $\mathcal H(a)$ are the Lie algebras with a Cartan matrix of the form $\begin{pmatrix} 2 & -a \\ -a &2 \end{pmatrix}$, $a\ge 1$. When $a=1$, the Lie algebra $\mathcal H(1)$ is nothing but $\frak {sl}_3(\mathbb C)$; when $a=2$, we obtain (a central extension of) the affine Lie algebra $\widehat{\frak{sl}_2}(\mathbb C)$. These Lie algebras are fundamental objects and their structures and representations are quite well-known. Surprisingly enough, when $a \ge 3$, we still do not know much about the Lie algebra $\mathcal H(a)$. What makes one intrigued is that there seem to be hidden connections of these algebras $\mathcal H (a)$ to automorphic forms.

Lepowsky and Moody \cite{LM} showed that there are remarkable connections between root systems of rank two (not necessarily symmetric) hyperbolic Kac-Moody algebras and quasi-regular cusps on Hilbert modular surfaces attached to certain quadratic fields. A. Feingold studied the algebra $\mathcal H(3)$ and described the root system of $\mathcal H(3)$ in terms of Fibonacci numbers \cite{Fein}. Kang and Melville extended this result to $\mathcal H(a)$, $a \ge 3$, using generalized Fibonacci numbers \cite{KaMe}. Furthermore, in the same paper, they studied root multiplicities of $\mathcal H(a)$, making use of Kang's formula for root multiplicities of Kac-Moody algebras \cite{Kang-94}. In 2004, Feingold and Nicolai showed that the algebras $\mathcal H (a)$ can be embedded into the rank three hyperbolic Kac-Moody algebra $\mathcal F$ associated with the Cartan matrix $\begin{pmatrix} 2 & -2 & 0 \\ -2 &2 &-1 \\0 &-1 &2 \end{pmatrix}$.

The Lie algebra $\mathcal F$ has connections to a Siegel modular form as shown in Feingold-Frenkel \cite{FF} and 
Gritsenko-Nikulin \cite{GN-96}. In particular, Gritsenko and Nikulin showed that the hyperbolic Kac-Moody algebra $\mathcal F$ is contained in a generalized Kac-Moody algebra $\mathcal G$ whose denominator function is a Siegel modular form, and called the generalized Kac-Moody algebra $\mathcal G$ an automorphic correction of $\mathcal F$. The notion of automorphic correction was originated from Borcherds' work \cite{Bor-92} on Monster Lie algebras. See Section \ref{correction} for the precise definition of automorphic correction due to Gritsenko and Nikulin \cite{GN-02}. 

The purpose of this paper is to investigate connections of the hyperbolic Kac-Moody algebras $\mathcal H(a)$, $a \ge 3$, to Hilbert modular forms from the point of view of automorphic correction. For each odd prime $p$, let $F=\Bbb Q(\sqrt{p})$, and to each positive solution of the Pell's equation $a^2-p s^2=4$, we associate a family of $\mathcal H(a)$'s. In section 5.1, we show that there exists a chain of embeddings in each family (Theorem \ref{thm-embedding}).

In particular, we consider three infinite families of $\mathcal H(a)$'s attached to the quadratic fields $\mathbb Q(\sqrt p)$, 
$p \in \{ 5, 13, 17 \}$, respectively. These three primes are the only primes for which there exists the unique weakly holomorphic modular form 
$f_m\in A_0^+(p,\chi_p)$ (See Section 4.1 for the definition of the space $A_0^+(p, \chi_p)$)
with the principal part $s(m)^{-1}q^{-m}$ for each $m\geq 1$, where 
$s(m)=\begin{cases} 1, &\text{if $p\nmid m$},\\ 2, &\text{if $p|m$}.\end{cases}$ 

Consider the first $\mathcal H(a)$ in each family, namely, 
$\mathcal H(3), \mathcal H(11), \mathcal H(66)$. In section 5.2, we show that there exists a generalized Kac-Moody superalgebra 
$\widetilde{\mathcal H}$ for each of these $\mathcal H(a)$'s, which contains the $\mathcal H (a)$ as a subalgebra, and whose denominator function is a Hilbert modular form $\Phi_1(z)$ for $\mathbb Q(\sqrt{p})$ (Theorem \ref{main}). Here the fact that $\Phi_1(z)$ is an infinite product, so-called Borcherds product, is crucial. Borcherds \cite{Bor-98} studied certain  lifts of weight $0$ weakly holomorphic modular forms to modular forms on orthogonal groups $O(2,2)$. Bruinier and Bundschuh \cite{BB} made explicit the correspondence between modular forms on $O(2,2)$ and Hilbert modular forms for $\Bbb Q(\sqrt{p})$, $p\equiv 1$ (mod 4). We use their explicit correspondence in showing that $\Phi_1(z)$ is indeed the automorphic correction of the denominator function of $\mathcal H(a)$. 
For $p=13, 17$, we also use the explicit calculation of Mayer \cite{M}. 

Let $f_m(z)=s(m)^{-1}q^{-m}+\sum_{n=0}^\infty a_m(n)q^n$. It is known \cite{BB} that $a_m(n)$ are rational numbers with bounded denominators. When $p=5,13$, more is true. Indeed we verify that $s(n)a_m(n)$ are integers for all $n$ (Lemma \ref{lem-513}). If $p=17$, it is likely that they are integers, but we were not able to verify it. We assume that they are integers. It is necessary since they are root multiplicities of the generalized Kac-Moody superalgebra 
$\widetilde {\mathcal H}$. 

In section 6, we apply the method of Hardy-Ramanujan-Rademacher \cite {L1} to calculate the asymptotics of the Fourier coefficients $a_m(n)$ (Theorem \ref{last}). In that way, we obtain information on the root multiplicities of $\widetilde {\mathcal H}$.

It is expected that our method can be applied to more general (not necessarily symmetric) rank two hyperbolic Kac-Moody algebras. We will consider these general cases in a subsequent paper.

\subsection*{Acknowledgments} We would like to thank V. Gritsenko for helpful discussions.

\vskip 1cm

\section{Rank Two Symmetric Hyperbolic Kac-Moody Algebras} \label{hyperbolic}

In this section we fix our notations for hyperbolic Kac-Moody algebras. A general theory of Kac-Moody algebras can be found 
in \cite{Kac}, and the rank two hyperbolic case was studied by Lepowsky and Moody \cite{LM}, Feingold \cite{Fein}, and Kang and Melville \cite{KaMe}.

Let $A=\begin{pmatrix} 2 & -a \\ -a & 2 \end{pmatrix}$ be a generalized Cartan matrix with $a \ge 3$, and  $\mathcal H(a)$ be the hyperbolic Kac-Moody algebra associated with the matrix $A$. In this section, we write $\mathfrak g = \mathcal H(a)$ if there is no need to specify $a$. Let $\{ h_1, h_2 \}$ be the set of simple coroots in the Cartan subalgebra $\mathfrak h = \mathbb C h_1 + \mathbb C h_2 \subset \mathfrak g$. Let $\{ \alpha_1 , \alpha_2 \} \subset \mathfrak h^*$ be the set of simple roots, and $Q=\mathbb Z \alpha_1 + \mathbb Z \alpha_2$ be the root lattice, and define $\mathfrak h^*_{\mathbb Q} =\mathbb Q \alpha_1 + \mathbb Q \alpha_2$ and $\mathfrak h^*_{\mathbb R} =\mathbb R \alpha_1 + \mathbb R \alpha_2$.  The set of roots of $\mathfrak g$ will be denoted by $\Delta$, and the set of positive (resp. negative) roots by $\Delta^+$ (resp. by $\Delta^-$), and the set of real (resp. imaginary) roots by $\Delta_{\mathrm{re}}$ (resp. by $\Delta_{\mathrm{im}}$). We will use the notation $\Delta^+_{\mathrm{re}}$ to denote the set of positive real roots. Similarly, we use $\Delta^+_{\mathrm{im}}$, $\Delta^-_{\mathrm{re}}$ and $\Delta^-_{\mathrm{im}}$.

We assume that $a^2-4=ps^2$ for some $s \in \mathbb N$ and an odd prime $p$, and let $F=\mathbb Q(\sqrt p)$. We denote by $\bar x$ the conjugate of $x \in F$ and write $N$ and $\mathrm{tr}$ for the norm and trace of $F$.
 We denote the ring of integers of $F$ by $\mathcal O$. By considering the Pell's equation
\begin{equation} \label{Pell} a^2 -ps^2 =4 ,\end{equation} we obtain infinitely many pairs $(a,s)$ for each $p$.
We set \[ \eta = \frac {a + \sqrt {a^2-4}} 2 = \frac {a+s\sqrt p} 2 .\] Then we have $\bar \eta = \eta^{-1}$ and $1+\eta^2=a \eta$.
If $p \equiv 1 \, (\mathrm{mod}\ 4)$, we fix a fundamental unit $\varepsilon_0$ of $F$ so that $\eta = \varepsilon_0^{2k}$ for some $k \in \mathbb N$. In this case $N(\varepsilon_0)=-1$ and $N(\eta)=1$. If $p \equiv 3 \, (\mathrm{mod}\ 4)$, we fix a fundamental unit $\varepsilon_0$ of $F$ so that $\eta = \varepsilon_0^{k}$ for some $k \in \mathbb N$. In this case $N(\varepsilon_0)=1$. For example, if $p=5$ then the smallest positive solution of the Pell's equation is $(a,s)=(3,1)$, if $p=13$ then $(a,s)=(11,3)$, and if $p=17$ then $(a,s)=(66,16)$, and we choose a fundamental unit $\varepsilon_0$ of $\mathcal O$ as follows:
\[ \varepsilon_0 = \frac {1+\sqrt 5} 2 \ \text{ for }p=5; \qquad \varepsilon_0= \frac {3+\sqrt {13}} 2 \ \text{ for }p=13; \qquad \varepsilon_0= 4+\sqrt {17} \ \text{ for }p=17.\]

The simple reflection corresponding to $\alpha_i$ in the root system of $\mathfrak g$ is denote by $r_i$ ($i=1, 2$), and the Weyl group by $W$. The eigenvalues of $r_1r_2$ as a linear transformation on $\mathfrak h^*$ are $\eta^2$ and $\eta^{-2}$. Let $\gamma^+$ be an eigenvector for $\eta^2$ and we set $\gamma^- =r_2 \gamma^+$. Then $\gamma^-$ is an eigenvector for $\bar \eta^{2}$. Specifically, we choose
\[ \gamma^+ = \frac {\alpha_1 + \bar \eta \alpha_2} s \quad \text{ and } \quad \gamma^- = \frac {\alpha_1 + \eta \alpha_2} s .\]
We define a symmetric bilinear form $(\cdot, \cdot)$ on $\mathfrak h^*$ to be given by the Cartan matrix $A$ with respect to $\{ \alpha_1, \alpha_2 \}$. Then we have $(\gamma^+, \gamma^+)=(\gamma^-, \gamma^-)=0$ and $(\gamma^+, \gamma^-)=-p$.

We will use the column vector notation for the elements in $\mathfrak h^*$ with respect to the basis $\{\gamma^+, \gamma^- \}$, i.e. we write $\binom  x y $ for $x \gamma^+ + y \gamma^-$. Then we have
\[ \alpha_1 = \frac 1 {\sqrt{p}} \begin{pmatrix} \eta \\ - \bar \eta \end{pmatrix} \quad \text{ and } \quad \alpha_2= \frac 1 {\sqrt{p}}  \begin{pmatrix} - 1 \\  1 \end{pmatrix} . \] It is now easy to see that $\mathfrak h^*_{\mathbb Q}= \{ \binom x {\bar x} \, | \, x \in F\}$.
 A symmetric bilinear form $\langle \cdot , \cdot \rangle$ on $F$ is defined by $\langle x, y \rangle= - p \, \mathrm{tr}(xy')$. We define a map $\psi : \mathfrak h^*_{\mathbb Q} \rightarrow F$ by $\binom  x {\bar x}  \mapsto x$. Then the map $\psi$ is an isometry from $(\mathfrak h^*_{\mathbb Q}, (\cdot, \cdot))$ to $(F, \langle \cdot, \cdot \rangle)$.
 In particular, the root lattice $Q=\mathbb Z \alpha_1 + \mathbb Z \alpha_2$ is mapped onto a sublattice of $\mathcal O/\sqrt p$. When $p \equiv 1$ (mod $4$), the inverse different $\frak d^{-1}$ is equal to $\mathcal O/{\sqrt p}$, and we have $s \frak d^{-1} \subset \psi (Q) \subset \frak d^{-1}$, and the dual lattice $(\frak d^{-1})'$ of $\frak d^{-1}$ is $\frac 1 p \mathcal O$.

Let $\omega_i$ $(i=1,2)$ be the fundamental weights of $\mathfrak g$. Then we have $\omega_1 = \frac 1 {4-a^2} ( 2 \alpha_1 + a \alpha_2)$ and $\omega_2 = \frac 1 {4-a^2}(a \alpha_1 + 2 \alpha_2)$. In the column vector notation,
\[ \omega_1 = \frac {-1} {sp} \begin{pmatrix} 1 \\ 1 \end{pmatrix} \quad \text{ and } \quad \omega_2= \frac {-1} {sp}  \begin{pmatrix} \eta \\ \bar \eta \end{pmatrix} . \] We define \[ \rho := - (\omega_1 + \omega_2) = \frac 1 {sp} \begin{pmatrix} 1+\eta \\ 1 + \bar \eta \end{pmatrix} . \]

 The simple reflections have the matrix representations
\[ r_1 = \begin{pmatrix} 0 &\eta^2 \\ \bar \eta^2 & 0 \end{pmatrix} \quad \text{ and } \quad r_2 = \begin{pmatrix} 0 & 1 \\ 1 & 0 \end{pmatrix} . \] The Weyl group $W$ also acts on $F$ by \[ r_1 x= \eta^2 \bar x \quad \text{ and } \quad r_2 x = \bar x \qquad \text{ for } x \in F,\] so that the isometry $\psi$ is $W$-equivariant.
Since $W=\{ (r_1r_2)^i, r_2(r_1r_2)^i \, | \, i \in \mathbb Z \}$, we can calculate the set of positive real roots and obtain
\[ \Delta^+_{\mathrm{re}} = \left \{ \frac 1 {\sqrt{p}} \begin{pmatrix}  \eta^{j} \\ - \bar \eta^{j} \end{pmatrix} (j > 0),  \qquad
\frac 1 {\sqrt{p}} \begin{pmatrix}  - \bar \eta^{j} \\  \eta^{j} \end{pmatrix} (j \ge 0)  \right \} . \] The set of imaginary roots is described in \cite{Fein, KaMe}. We present it using our notations: First, we define the set
\[ \Omega_k = \left \{ (m,n) \in \mathbb Z_{\ge 0} \times \mathbb Z_{\ge 0} : \sqrt{\frac {4 k}{a^2 -4} } \le m \le \sqrt {\frac k {a-2} },\ n = \frac { am-\sqrt{(a^2-4)m^2 -4k}} 2 \right \} \]  for $k \ge 1$.
Note that we have only to present $\psi(\Delta^+_{\mathrm{im}})$. The set is given by
\begin{equation} \label{imroot} \psi(\Delta^+_{\mathrm{im}}) = \left \{  \frac 1 {\sqrt{p}} \, \eta^{j} (m \eta -n) ,  \
\frac 1 {\sqrt{p}}  \, \eta^{j}(n \eta -m) , \
\frac 1 {\sqrt{p}} \, \bar \eta^{j}(n - m \bar \eta ) , \
\frac 1 {\sqrt{p}} \,  \bar \eta^{j}(m  - n \bar \eta) \right \} , \end{equation}
where $j \ge 0$ and $(m,n) \in \Omega_k$ for $k \ge 1$.

\vskip 1 cm

\section{Modular forms on $O(2,2)$ as Hilbert modular forms} \label{O22}

In this section, we review the result of \cite{B} on the correspondence between Hilbert modular forms and modular forms on $O(2,2)$ in a special case after we consider the general case of modular forms on $O(n,2)$.

\subsection{Modular forms on $O(n,2)$}

Let $(V,Q)$ be a non-degenerate quadratic space over $\mathbb Q$ of type $(n,2)$.
Let $V(\Bbb C)$ be the complexification of $V$ and $P(V(\Bbb C))=(V(\Bbb C)-\{0\})/\Bbb C^*$ be the corresponding projective space.
Let $\mathcal{K}^+$ be a connected component of
\begin{equation} \label{zz} \mathcal{K}=\{ [Z]\in P(V(\Bbb C)) : (Z,Z)=0,\, (Z,\bar Z)<0\},
\end{equation}
and let $O^+_V(\mathbb R)$ be the subgroup of elements in $O_V(\mathbb R)$ which preserve the components of $\mathcal K$.

For $Z\in V(\Bbb C)$, write $Z=X+iY$ with $X,Y\in V(\Bbb R)$.
Given an even lattice $L\subset V$, let $\Gamma \subseteq O^+_L:=O_L\cap O_V^+(\Bbb R) $ be a subgroup of finite index. Then $\Gamma$ acts on $\mathcal{K}$ discontinuously.
Let
$$\widetilde{\mathcal{K}}^+=\{ Z\in V(\Bbb C)-\{0\} : [Z]\in \mathcal{K}^+\}.
$$

Let $k\in  \frac 1 2 \mathbb Z$, and $\chi$ be a multiplier system of $\Gamma$. Then a meromorphic function $\Phi: \tilde{\mathcal{K}}^+\longrightarrow \Bbb C$ is called a {\em meromorphic modular form} of weight $k$ and
multiplier system $\chi$ for the group $\Gamma$, if
\begin{enumerate}
\item $\Phi$ is homogeneous of degree $-k$, i.e., $\Phi(cZ)=c^{-k}\Phi(Z)$ for all $c\in\Bbb C-\{0\}$,

\item $\Phi$ is invariant under $\Gamma$, i.e., $\Phi(\gamma Z)=\chi(\gamma)\Phi(Z)$ for all $\gamma\in \Gamma$.
\end{enumerate}
This definition agrees with the one given in \cite{GN-02}.

\subsection{Hilbert modular forms on quadratic number fields} \label{subsec-H}

For a prime $p \equiv 1$ (mod $4$), let $F=\Bbb Q[\sqrt{p}]$, and let $\mathcal{O}_F$ and $\frak{d}_F$ be the ring of integers and the different of $F$, respectively. We denote by $\bar x$ the conjugation of $x$ in $F$.
We set $\Gamma_F=SL_2(\mathcal{O}_F)$. Assume that $\Gamma \subseteq \Gamma_F$ is a subgroup of finite index. Let $\mathbb H$ be the upper-half plane and $\chi$ be a multiplier system of $\Gamma$.
A meromorphic function $f: \Bbb H^2\longrightarrow \Bbb C$ is called a {\em Hilbert modular form} of weight $k$ for $\Gamma$ if
$$f(\gamma z)=\chi(\gamma)N(cz+d)^k f(z),
$$
for $\gamma=\begin{pmatrix} a&b\\ c&d\end{pmatrix}\in\Gamma$, and $N(cz+d)=(cz_1+d)(\bar cz_2+\bar d)$ for $z=(z_1,z_2) \in \mathbb H^2$.

Consider the $\mathbb Q$-vector space $V=\Bbb Q\oplus \Bbb Q\oplus F$, and define a quadratic form $Q$ on $V$ by \[ Q(a,b,\nu)=-p \, (\nu \bar \nu+ab) \] and a bilinear form $B$ so that $B((a, b, \nu), (a,b, \nu))=2Q(a,b,\nu)$. Then
$(V,Q)$ is a quadratic space of type $(2,2)$.
We will consider the lattice $L=\Bbb Z\oplus\Bbb Z\oplus {\frak d}^{-1}_F$ in $V$.

Let
$$\tilde V=\{X\in M_2(F): X^t=\bar X \}= \left \{ \begin{pmatrix} a& \nu \\ \bar \nu&b\end{pmatrix}: \, a,b\in\Bbb Q, \ \nu\in F \right \}.
$$
Then $\tilde V$ is a rational quadratic space with the quadratic form $\tilde Q(X)=p \, \det(X)$. The corresponding bilinear form is $\tilde B(X_1,X_2)=p \, \mathrm{tr}(X_1 X_2^*)$,
where $X^*=\begin{pmatrix} d& -b \\ -c& a\end{pmatrix}$ for $X=\begin{pmatrix} a&b \\ c&d\end{pmatrix}$.
Here $SL_2(F)$ acts on $\tilde V$ by $X\mapsto g.X=gX{\bar g}^t$ for $X \in \tilde V$ and $g \in SL_2(F)$.
Then $\tilde V$ and $V$ are isometric with the isometry given by
\begin{equation} \label{eqn-metry} \tilde V\longrightarrow V,\quad \begin{pmatrix} a& \nu \\ \bar \nu&b\end{pmatrix}\mapsto \left (-a,b, \nu \right ).
\end{equation}
Under the isomorphism, we have
\begin{equation} \label{eqn-latt} L=\Bbb Z\oplus\Bbb Z\oplus\frak{d}_F^{-1}\simeq  \left \{ \begin{pmatrix} a& \nu \\ \bar \nu&b\end{pmatrix}\in \tilde V: a,b\in\Bbb Z, \ \nu\in \mathfrak d_F^{-1} \right \}.
\end{equation}
Note that the dual lattice $L'$ is given by
\[ L' = \tfrac 1 p \mathbb Z \oplus \tfrac 1 p \Bbb Z \oplus \tfrac 1 p \mathcal O. \]

The two real embeddings $F\rightarrow \Bbb R^2$,  $x\mapsto (x,\bar x)$, induces an embedding $\tilde V\mapsto M_2(\Bbb R)$. Thus we have
$\tilde V(\Bbb C)=M_2(\Bbb C)$, and let
$$\mathcal{K}=\{ [Z]\in P(M_2(\Bbb C)) : \det(Z)=0,\, \mathrm{tr}(Z {\bar Z}^*)<0\}.
$$
We write $M(z)=\begin{pmatrix} z_1z_2&z_1\\z_2&1\end{pmatrix}\in M_2(\Bbb C)$ for $z=(z_1, z_2) \in \mathbb C^2$. Note that $[M(z)]\in\mathcal{K}$ if and only if $Im(z_1)Im(z_2)>0$. Let $\mathcal{K}^+$ be a connected component of $\mathcal{K}$. Then
$\Bbb H^2\rightarrow \mathcal{K}^+$, $z=(z_1,z_2)\mapsto [M(z)]$, is a biholomorphic map.
For $\gamma=\begin{pmatrix} a&b\\ c&d\end{pmatrix}\in SL_2(F)$, we have
$$\gamma M(z)=N(cz+d) M(\gamma z).
$$
One can easily see that $\Gamma_F \subset O^+_L$.
Therefore, modular forms of weight $k$ on $O(2,2)$  can be considered as Hilbert modular forms of weight $k$.

\vskip 1 cm

\section{Automorphic Correction} \label{correction}

In this section, we recall the theory of automorphic correction established by Gritsenko and Nikulin \cite{GN-96,GN-97,GN-02}. The original idea of automorphic correction can be traced back to Borcherds' work \cite{Bor-92}.

We assume that the following data (1)-(4) are given.
\begin{enumerate}
\item We are given a lattice $M$ with a non-degenerate integral symmetric bilinear form $(\cdot, \cdot)$ of signature $(n,1)$ for some $n \in \mathbb N$.

\item A nontrivial reflection group $W \subset O(M)$ is given. The group $W$ is generated by reflections in some roots of the lattice $M$. A vector $\alpha \in M$ is called a root if $(\alpha, \alpha) >0$ and $(\alpha, \alpha)$ divides $2(\alpha , \beta)$ for all $\beta \in M$.

\item Consider the cone \[ V(M)= \{ \beta \in M \otimes \mathbb R \, | \, (\beta,\beta) <0 \} ,\] which is a union of two half cones. One of these half cones is denoted by $V^+(M)$. Choose a minimal set $\Pi$ of roots orthogonal to a fundamental chamber $\mathcal M \subset V^+(M)$ of $W$ so that \[ \mathcal M = \{ \beta \in V^+(M) \, | \, (\beta, \alpha) \le 0  \text{ for all } \alpha \in \Pi \} .\] Moreover, we have a Weyl vector $\rho \in M \otimes \mathbb Q$ satisfying $(\rho , \alpha) = - (\alpha, \alpha)/2$ for each $\alpha \in \Pi$.

\item Define the complexified cone $\Omega(V^+(M))= M \otimes \mathbb R + i V^+(M)$. Let $L=\begin{pmatrix} 0 & -m \\ -m & 0 \end{pmatrix} \oplus M$ be an extended lattice for some $m \in \mathbb N$. We consider the quadratic space $V=L \otimes \mathbb Q$ and obtain $\mathcal K^+$ as in \eqref{zz}. Define a map $\Omega(V^+(M)) \rightarrow \mathcal K$ by $z \mapsto \left [ \frac {(z, z)} {2m} \, e_1 +  e_2 + z \right ]$, where $\{ e_1, e_2 \}$ is the basis for $\begin{pmatrix} 0 & -m \\ -m & 0 \end{pmatrix}$. Then the space $\mathcal K^+$ is canonically identified with $\Omega(V^+(M))$.
We are given a holomorphic automorphic form $\Phi(z)$ on $\Omega(V^+(M))$ with respect to a subgroup $\Gamma \subset O^+_L$ of finite index.
    The automorphic form $\Phi$ has a Fourier expansion of the form
     \[ \Phi(z) = \sum_{w \in W} \det (w) \left ( e \left ( -  (w (\rho), z) \right )- \sum_{a \in M \cap
     \mathcal M} m(a) \, e(- (w(\rho+a), z)) \right ), \] where  $e(x)=e^{2 \pi i x}$ and $m(a) \in \mathbb Z$ for
     all $a \in M \cap \mathcal M$.
\end{enumerate}

 The matrix
    \[ A= \begin{pmatrix} \frac {2(\alpha, \alpha')}{(\alpha, \alpha)} \end{pmatrix}_{\alpha, \alpha' \in \Pi} \]
defines a Kac-Moody algebra $\mathfrak g$. Moreover, the data (1)-(4) define a generalized Kac-Moody superalgebra $\mathcal G$ as in \cite{GN-02}. We call $\mathcal G$ (or $\Phi(z)$) an {\em automorphic correction} of $\mathfrak g$.
 The automorphic form $\Phi(z)$ determines the set of simple imaginary roots of $\mathcal G$, and can be written, using the denominator identity for the generalized Kac-Moody superalgebra $\mathcal G$, as the product
 \[ \Phi(z)= e(-(\rho, z)) \prod_{\alpha \in \Delta(\mathcal G)^+} (1 - e(-(\alpha, z)))^{\mathrm{mult}(\mathcal G, \alpha)} ,\] where $\Delta(\mathcal G)^+$ is the set of positive roots of $\mathcal G$ and $\mathrm{mult}(\mathcal G, \alpha)$ is the root multiplicity of $\alpha$ in $\mathcal G$.
  In Section \ref{H-correct}, we will construct automorphic corrections of the hyperbolic Kac-Moody algebras $\mathcal H(3)$, $\mathcal H(11)$, $\mathcal H(66)$ which are associated with the quadratic fields $\mathbb Q(\sqrt p)$, $p=5, 13, 17$, respectively.

\vskip 1 cm

\section{Hilbert Modular Forms as Borcherds Products} \label{Hilbert}

In this section we summarize the results of \cite{BB}. (cf. \cite{M})
\subsection{Weakly holomorphic modular forms of weight 0} \label{weakly}
Let $p$ be an odd prime and $A_k^+(p,\chi_p)$ (resp. $A_k^-(p, \chi_p)$) be the space of weakly holomorphic modular forms $f$ of weight $k$ for the group $\Gamma_0(p)$ with character $\chi_p$ such that
$a(n)=0$ if $\chi_p(n)=-1$ (resp. $\chi_p(n)=1$) , where
$f=\sum_{n\in\Bbb Z} a(n)q^n$ and $\chi_p(n)=(\frac np)$. We denote by $S_k^+(p, \chi_p)$ (resp. $S_k^-(p, \chi_p)$) the subspace of cusp forms.
For an integer $n$, define $s(n)=\begin{cases} 2 &\text{if $p|n$}, \\ 1  &\text{otherwise.}\end{cases}$

\begin{Thm} \cite[Theorem 6]{BB}There exists a weakly holomorphic modular form $f\in A_0^+(p,\chi_p)$ with prescribed principal part $\sum_{n<0} a(n)q^n$  if and only if
$\sum_{n<0} s(n)a(n)b(-n)=0$ for every cusp form $g=\sum_{m>0} b(m)q^m\in S_2^{\delta}(p,\chi_p)$, where $\delta=\chi_p(-1)$.
\end{Thm}

In the rest of this section, we assume that $p \in \{ 5, 13, 17 \}$. Then we have $S_2^+(p,\chi_p)=0$. For a given positive integer $m$ with $\chi_p(m)\ne -1$, we let
$$f_m=\sum_{n\geq -m} a_m(n) q^n=s(m)^{-1} q^{-m}+\sum_{n=0}^\infty a_m(n) q^n,
$$
be the unique element of $A_0^+(p,\chi_p)$, whose principal part is $s(m)^{-1}q^{-m}$.

When $p=5$,
\begin{eqnarray*}
f_1 &=& q^{-1}+5+11q-54q^4+55q^5+44q^6-395q^9+340q^{10}+\cdots, \\
f_4 &=& q^{-4} +15-216q+4959q^4+22040q^5-90984q^6+\cdots,\\
f_9 &=& q^{-9}+35-3555q+922374q^4+7512885q^5-53113164q^6+\cdots.
\end{eqnarray*}
When $p=13$,
\begin{eqnarray*}
f_1 &=& q^{-1}+1+q+3q^3-2q^4-q^9-4q^{10}+10q^{12}+\cdots,  \\
f_4 &=& q^{-4} +3- 8q+16q^3+29q^4-70q^9-2q^{10}-32q^{11}+\cdots,\\
f_9 &=& q^{-9}+13- 9q+36q^3-198q^4+\cdots.
\end{eqnarray*}
When $p=17$,
\begin{eqnarray*}
f_1 &=& q^{-1}+\frac 12 -q+q^2+2q^4-q^8-2q^9+q^{13}-q^{15}+2q^{16}+\cdots, \\
f_4 &=& q^{-4} +\frac 72+8q-2q^2+11q^4-5q^8+16q^9-56q^{13}+\cdots,\\
f_9 &=& q^{-9}+ \frac 72-18q-27q^2+36q^4+243q^8+41q^9-279q^{13}+\cdots.
\end{eqnarray*}

If $p=5, 13$, we can prove that $f_1$ has integer coefficients. This follows from the fact that \[f_1(z)=\frac {E_2^{(p)}(z)}{H_2(z)},\] where $E_2^{(p)}$ is the normalized Eisenstein series of weight 2 for $\Gamma_0(p)$ with the trivial character and $H_2$ is the Eisenstein series with the character $\chi_p$ corresponding to the cusp 0 (there is a typo in \cite{M}, page 114):
\begin{eqnarray*} E_2^{(p)}(z) &=& 1 + \frac {24}{p-1}\sum_{n=1}^\infty (\sigma(n)-p\,  \sigma(n/p)) q^n, \\
H_2(z) &=& \sum_{n=1}^\infty (\sum_{d|n} d \chi_p(n/d)) q^n=q+O(q^2).
\end{eqnarray*}
Here we put the convention that if $p\nmid n$, $\sigma(n/p)=0$.
Note that if $p=5$, $H_2(z)=H^{(q)}(z)=\frac {\eta(5z)^5}{\eta(z)}$. When $p=13$, $H_2(z)$ may have a zero in the upper half plane, and yet it is remarkable that the quotient $\frac {E_2^{(p)}(z)}{H_2(z)}$ does not have a pole in the upper half plane.

In general, it is known that $a_m(n)$ are rational numbers and have a bounded denominator. See Proposition 8 in \cite{BB}.
But we can see that more is true in the case of $p=5, 13$.

\begin{Lem} \label{lem-513} Let $p=5, 13$. Then $s(n)a_m(n)$ are integers for all $n\geq 0$.
\end{Lem}
\begin{proof} Let $\tilde H(z)=\frac {\eta(z)^k}{\eta(pz)^k}$, where $k=\frac {24}{gcd(24,p-1)}$. It belongs to $A_0(p,1)$.
In \cite{M}, page 112, we can see that $f_m(z)$, $m<p$, are generated by $\tilde H(z)$ and $f_1(z)$ over $\Bbb Z$. Hence they have integer coefficients. On the other hand, $f_p$ is obtained from $\frac 12 E_0$ by subtracting suitable integer multiples of $f_m$, $m<p$.
Hence it is enough to observe that if $E_0(z)=q^{-p}+\sum_{n=-p+1}^\infty b(n)q^n$, $b(n)$ is an even integer for $p\nmid n$. 
Recall that 
$$E_0(z)=E_2^+(z) \frac {E_4E_6}{\Delta}(pz), \quad E_2^+(z)=1+\frac 2{L(-1,\chi_p)} \sum_{n=1}^\infty \sum_{d|n} d(\chi_p(d)+\chi_p(n/d)) q^n .
$$
Here $L(-1,\chi_5)=-\frac 25$, and $L(-1,\chi_{13})=-2$, and if $p\nmid n$, then $p\nmid d$ for any $d|n$. So the possible values of $\chi_p(d)+\chi_p(n/d)$ are $0, \pm 2$. Therefore, we can write $E_2^+(z)=1+2X+Y$, with $X=\sum_{n=1,\  p \nmid n }^\infty A(n)q^n$ and $Y= \sum_{n=1}^\infty B(n) q^{pn}$. On the other hand, $\frac {E_4E_6}{\Delta}(pz)=q^{-p}(1+Z)$, with $Z=\sum_{n=1}^\infty C(n)q^{pn} $,  where $A(n), B(n), C(n)$ are integers. Hence our assertion is clear.

If $m>p$, $f_m(z)$ can be obtained from $j(pz) f_{m-p}(z)$ by subtracting suitable integer multiples of $f_{m'}$, $m'<m$. Hence by induction, we can see that for each $m$, $s(n)a_m(n)$ are integers.
\end{proof}

When $p=17$, it is likely that $s(n)a(n)$ are integers for $f_1=q^{-1} + \sum_{n\ge0} a(n) q^n$. But we were not able to verify it.
In Section \ref{H-correct}, we will assume that $s(n)a(n)$ are integers for all $n \ge 1$. 

\subsection{Borcherds lifts}
Let  $p \in \{ 5, 13, 17 \}$ and $F=\Bbb Q[\sqrt{p}]$. Denote the ring of integers of $F$ by $\mathcal{O}=\Bbb Z[\frac {1+\sqrt{p}}2]$ and the different of $F$ by $\frak{d}=(\sqrt{p})$. We keep the fundamental units $\varepsilon_0=\frac {1+\sqrt{5}}2$ for $p=5$,
$\varepsilon_0=\frac {3+\sqrt{13}}2$ for $p=13$,
$\varepsilon_0=4+\sqrt{17}$ for $p=17$ as in Section \ref{hyperbolic}.
Let $(z_1,z_2)$ be a standard variable on $\Bbb H^2$ and write $(y_1,y_2)$ for its imaginary part.
The Hilbert modular group $\Gamma_F=SL_2(\mathcal{O})$ acts on $\Bbb H^2$ in the usual way.
For a positive integer $m$ with $\chi_p(m)\ne -1$, let
$$S(m)=\bigcup_{\substack{\lambda\in\frak{d}^{-1} \\  N(\lambda)=-\frac mp }} \{ (z_1,z_2)\in \Bbb H^2:\, \lambda y_1+ \bar \lambda y_2=0\}.
$$

Let $f=\sum_{n\in\Bbb Z} a(n)q^n\in A_0^+(p,\chi_p)$ and assume that $s(n)a(n)\in\Bbb Z$ for all $n<0$.
Let $\mathcal W\subset \Bbb H^2$ be a Weyl chamber attached to $f$, i.e., a connected component of
$$
\Bbb H^2- \bigcup_{\substack{ n<0 \\  a(n)\ne 0 }}  S(-n).
$$
For $\lambda\in\frak{d}^{-1}$, we write $(\lambda, \mathcal W)>0$ if $\lambda \, y_1+\bar \lambda \, y_2>0$ for all $(z_1,z_2)\in \mathcal W$. Put $N=\min\{ n : a(n)\ne 0\}$. Then we have:
\begin{Thm} \cite{Bor-98, BB} \label{thm-lift} The Borcherds lift of $f$ is given by
$$
\Psi(z_1,z_2)=e(\rho_{\mathcal W} \, z_1+\overline{\rho_{\mathcal W}}\, z_2)\prod_{\substack{ \nu\in\frak{d}^{-1} \\  (\nu,\mathcal W)>0 } } (1-e(\nu z_1+\bar \nu z_2))^{s(p\nu\bar \nu) a(p\nu\bar \nu)}.
$$
Here $\Psi(z_1,z_2)$ is a Hilbert modular form of weight $a(0)$, and the product converges absolutely for all $(z_1,z_2)$ with $y_1y_2>\frac {|N|}p$ outside the set of poles. (See below for the definition of $\rho_{\mathcal W}$).
\end{Thm}

The Weyl vector $\rho_{\mathcal W}$ and its conjugate $\overline{\rho_{\mathcal W}}$ are contained in $(\mathrm{tr}(\varepsilon_0))^{-1}\frak{d}^{-1}$, where $\varepsilon_0>0$ is the fundamental unit of $F$. More precisely, the vectors $\rho_{\mathcal W}$ and $\overline{\rho_{\mathcal W}}$ are given as follows:
For a negative integer $n$ with $a(n)\ne 0$, define
$$R(\mathcal W,n)=\{\lambda\in\frak{d}^{-1}: \, \lambda>0,\, N(\lambda)=\tfrac np,\, \lambda y_1+\bar\lambda y_2 < 0,\, \varepsilon_0^2\lambda y_1+{\bar\varepsilon_0}^2 \bar\lambda \, y_2>0,\, \text{for all $(z_1,z_2)\in \mathcal W$}\}.
$$
Then
\begin{equation} \label{eqn-rho} \rho_{\mathcal W} \, y_1+\overline{\rho_{\mathcal W}} \, y_2=\sum_{n<0} s(n)a(n) \frac 1{\mathrm{tr}(\varepsilon_0)} \sum_{\lambda\in R(\mathcal W,n)} (\varepsilon_0\lambda y_1+\bar\varepsilon_0 \bar\lambda \, y_2).
\end{equation}

Let $\Psi_m$ be the Borcherds lift of $f_m=s(m)^{-1}q^{-m}+\sum_{n=0}^\infty a(n)q^n$.  Write $m=q_1^{k_1}\cdots q_r^{k_r}$ into the prime factorization with distinct primes $q_i$. First, assume that $( \frac {q_i} p )=-1$ with an odd $k_i$ for some $i$. Then $m$ cannot be the norm of an ideal in $\mathcal{O}$,  and $S(m)$ is empty, and $\Bbb H^2$ is the only Weyl chamber for $f_m$. In this case, $\rho_{\mathcal W}=0$ and Theorem \ref{thm-lift} yields
\begin{equation} \label{eqn-Bor-1} \Psi_m(z_1,z_2)=\prod_{\substack{\nu\in\frak{d}^{-1} \\ \nu\gg 0}} (1-e(\nu z_1+\bar \nu z_2))^{s(p\nu\bar \nu)a(p\nu\bar \nu)}.
\end{equation}

Now we assume that  $( \frac {q_i} p )\neq 1$ for all $i$ and consider $\Psi_{m^2}$ for $m\in \mathbb N$. For example, if $p=17$, we consider $\Psi_1$ and $\Psi_9$; however we do not consider $\Psi_4$ or $\Psi_{16}$ since $(\frac 2 {17}) =1$.
Then $S(m^2)$ is not empty. More importantly, since no $q_i$ splits in $F$, the condition $N(\nu)=-m^2/p$ implies that $\nu=\pm \frac m {\sqrt p} \varepsilon_0^{2j}$ for some $j \in \mathbb Z$.
Let $\mathcal W$ be the Weyl chamber attached to $f_{m^2}$ that contains the point $(\sqrt{-1}, 2 \sqrt{-1})$.
Then $R(\mathcal W,-m^2)=\{\frac m{\sqrt{p}}\}$ and no other elements are included due to the condition on $m$. Hence, we obtain from \eqref{eqn-rho}
\[ \rho_{\mathcal W} = \frac {m \varepsilon_0 z_1}{\mathrm{tr}(\varepsilon_0)\sqrt{p}} \qquad \text{ for } m \in \mathbb N .\]

By \cite{M}, page 82, $(\nu,\mathcal W)>0$ is equivalent to $(\nu,\tau)>0$ for a point $\tau\in \mathcal W$. So in our case, it is equivalent to
$\nu + 2 \bar \nu >0$. If $\nu \gg 0$ (and $\nu + 2 \bar \nu >0$) for $\nu \in \frak d^{-1}$ then $N(\nu)>0$. If $\nu \not\gg 0$ and $\nu + 2 \bar \nu >0$, then $N(\nu)<0$ and $a(p\nu \bar\nu) \neq 0$ only for $\nu$ with $N(\nu)=-m^2/p$, in which case $s(p\nu \bar\nu)a(p\nu \bar\nu)=1$.
Therefore,
\begin{eqnarray}
\Psi_{m^2}(z_1,z_2) &=& e\left( \frac {m\varepsilon_0 z_1}{\mathrm{tr}(\varepsilon_0)\sqrt{p}}-\frac {m \bar \varepsilon_0 z_2}{\mathrm{tr}(\varepsilon_0)\sqrt{p}}\right) \prod_{\substack{\nu\in\frak{d}^{-1} \\ \nu+ 2\bar\nu>0} }
(1-e(\nu z_1+\bar \nu z_2))^{s(p\nu\bar \nu)a(p\nu\bar \nu)}  \nonumber \\
 &=& e\left( \frac {m\varepsilon_0 z_1}{\mathrm{tr}(\varepsilon_0)\sqrt{p}}-\frac {m \bar \varepsilon_0 z_2}{\mathrm{tr}(\varepsilon_0)\sqrt{p}}\right)\prod_{\substack{\nu\in\frak{d}^{-1} \\ \nu \gg 0} }
(1-e(\nu z_1+\bar \nu z_2))^{s(p\nu\bar \nu)a(p\nu\bar \nu)} \nonumber \\ & & \phantom{LLLLLLLLLLL} \times \prod_{\substack{\nu\in\frak{d}^{-1} ,\ \nu+ 2\bar\nu>0 \\ N(\nu)=-m^2/p} }
(1-e(\nu z_1+\bar \nu z_2)) . \label{eqn-Bor-2}
\end{eqnarray}

\vskip 1 cm

\section{Embedding of Hyperbolic Kac-Moody Algebras} \label{sec-embedding}

In this section we associate a family of $\mathcal H(a)$'s to each odd prime $p$ and prove that there exists a chain of embeddings among the algebras in each family. When $p=5, 13$ or $17$, we construct an automorphic correction of the first $\mathcal H(a)$ in each family, i.e. $\mathcal H(3)$, $\mathcal H(11)$, $\mathcal H(66)$ for $p=5, 3, 17$, respectively. The automorphic correction will be given by the Hilbert modular form $\Psi_1$ considered in the previous section. We also consider other $\Psi_m$ ($m \neq 1$) and see where the obstructions are for this to be an automorphic correction.

\subsection{Embedding of $\mathcal H (a)$}
We fix an odd prime $p$. We consider the Pell's equation \eqref{Pell} again. Recall that we fixed a fundamental unit $\varepsilon_0$ of $F$. We enumerate the solutions $(a_k, s_k)$ ($k=1, 2, \ldots$) of the equation so that if $p \equiv 1 \ (\mathrm{mod}\ 4)$,
\[ \eta_k = \frac {a_k + s_k \sqrt p} 2 = \varepsilon_0^{2k}, \qquad k=1, 2, \ldots,\] and
if $p \equiv 3 \ (\mathrm{mod}\ 4)$,
\[ \eta_k = \frac {a_k + s_k \sqrt p} 2 = \varepsilon_0^{k}, \qquad k=1, 2, \ldots.\]
Note that we have \[\eta_k^j = \eta_{kj} \quad \text{ for } k,j \in \mathbb N.\]

In Section \ref{hyperbolic}, we established an isometry of $\mathfrak h^*_{\mathbb Q}$ to $F$ for each $\mathcal H(a_k)$
  and obtained the set of positive real roots of $\mathcal H(a_k)$. Since the isometry depends on $k$, we denote it by $\psi_k$. Then we have
\begin{equation} \label{eqn-re} \psi_k(\Delta^+_{\mathrm{re}}) = \left \{  \frac 1 {\sqrt{p}} \, \eta_k^{j} \ (j > 0), \qquad  - \frac 1 {\sqrt{p}} \,  \bar \eta_k^{j}  \ (j \ge 0) \right \} . \end{equation}
We call an element of $\psi_k(\Delta^+_{\mathrm{re}})$ a positive real root by abusing the terminology. The set $\psi_k(\Delta^+_{\mathrm{im}})$ is given in \eqref{imroot}.
We will apply the following proposition to establish embeddings of $\mathcal H(a_k)$.

\begin{Prop}[\cite{FN}] \label{FN}
Let $\Delta$ be the set of roots of a Kac-Moody algebra $\mathfrak g$, with Cartan subalgebra $\mathfrak h$, and let $\Delta^+_{\mathrm{re}}$ be the set of positive real roots of $\mathfrak g$. Let $\beta_1, \cdots , \beta_n \in \Delta^+_{\mathrm{re}}$ be chosen such that for all $1 \le i\neq j \le n$, we have $\beta_i -\beta_j \notin \Delta$. For
$1 \le i \le n$, let $E_i$ and $F_i$ be nonzero root vectors in the root spaces corresponding to $\beta_i$ and $-\beta_i$, respectively, and let $H_i=[E_i, F_i] \in \mathfrak h$. Then the Lie subalgebra of $\mathfrak g$ generated by $\{E_i, F_i, H_i \,|\, 1 \le i \le n \}$ is a Kac-Moody algebra with Cartan matrix $\begin{pmatrix} \frac {2(\beta_i, \beta_j)}{(\beta_j, \beta_j)} \end{pmatrix}_{1\le i,j \le n}$.
\end{Prop}

We fix $k$  for the time being. Consider two positive real roots of $\mathcal H (a_k)$:
\begin{equation} \label{even} \beta_1 = \frac 1 {\sqrt p} \, \eta_k^{j} \quad \text{ and } \quad \beta_2 =  -\frac 1 {\sqrt p} \, \bar \eta_k^{j}  \quad  \text{ for } j>0 . \end{equation} Since $\beta_1 -\beta_2 = \frac 1 {\sqrt p} ( \eta_k^{j} + \bar \eta_k^{j} ) = \frac 1 {\sqrt p} (\eta_{kj}+ \bar \eta_{kj}  ) = \frac 1 {\sqrt p} a_{kj}$, it is clear that $\beta_1 -\beta_2$ is not a root. We also see that $\langle \beta_i, \beta_i \rangle =2$ $(i=1,2)$ and
\[ \langle \beta_1, \beta_2 \rangle = -p \left (   \tfrac 1 {\sqrt p} \,  \eta_k^{j}  \tfrac 1 {\sqrt p} \, \eta_k^{j} + \tfrac 1 {\sqrt p} \, \bar \eta_k^{j}  \tfrac 1 {\sqrt p} \, \bar \eta_k^{j} \right )= -(   \eta_k^{2j} +\bar \eta_k^{2j}) = - a_{2kj}  .\]
Similarly, if we take \begin{equation} \label{odd}  \beta_1 =  \frac 1 {\sqrt p} \, \eta_k^{j} \quad \text{ and } \quad \beta_2 =  -\frac 1 {\sqrt p} \, \bar \eta_k^{j-1}  \quad  \text{ for } j>0 ,\end{equation} then $\beta_1 -\beta_2 = \frac 1 {\sqrt p} ( \eta_k^{j}+\bar \eta_k^{j-1}) = \frac 1 {\sqrt p} \bar \eta_k^{j-1} (1 + \eta_{k}^{2j-1})$. Comparing it with elements in $\psi_k(\Delta^+_{\mathrm{re}})$ and $\psi_k(\Delta^+_{\mathrm{im}})$, we see that $\beta_1 -\beta_2$ is not a root. We also have that $\langle \beta_i, \beta_i \rangle =2$ $(i=1,2)$ and
\[ \langle \beta_1, \beta_2 \rangle = -p \left (  \tfrac 1 {\sqrt p} \,  \eta_k^{j-1}  \tfrac 1 {\sqrt p} \, \eta_k^{j} + \tfrac 1 {\sqrt p} \, \bar \eta_k^{j-1}  \tfrac 1 {\sqrt p} \, \bar \eta_k^{j} \right )= -(  \eta_k^{2j-1}+\bar \eta_k^{2j-1}) = - a_{k(2j-1)}  .\]

Hence we obtain the following theorem.

\begin{Thm} \label{thm-embedding}
Let $k$ and $l$ be positive integers, and assume that $k\,|\,l$.  Then there exists an embedding of $\mathcal H(a_l)$ into $\mathcal H(a_k)$ as a Lie subalgebra. Moreover, the root space of $\beta$ in $\mathcal H(a_l)$ is embedded into the root space of $\alpha$ in $\mathcal H(a_k)$ so that $\psi_l(\beta) =\eta_{kj}\psi_k(\alpha)$ if $l=2kj$ for some $j\in \mathbb N$ and $\psi_l(\beta) =\eta_{k(j-1)}\psi_k(\alpha)$ if $l=k(2j-1)$ for some $j \in \mathbb N$.
\end{Thm}

\begin{proof}
Applying Proposition \ref{FN} to the above computations, we obtain the first assertion. For the second assertion, we have only to investigate the simple roots. We notice that the simple roots of $\mathcal H(a_l)$ are $\frac 1 {\sqrt p} \eta_l$ and $-\frac 1 {\sqrt p}$. Assume that $l=2kj$. If we multiply the simple roots of $\mathcal H(a_l)$ by $\bar \eta_{kj}$, we obtain $\frac 1 {\sqrt p} \eta_{kj}$ and $-\frac 1 {\sqrt p} \bar \eta_{kj}$, which are the roots in \eqref{even} and generate a copy of $\mathcal H(a_l)$ inside $\mathcal H(a_k)$. Now assume that $l=k(2j-1)$. Multiplying the simple roots of $\mathcal H(a_l)$ by $\bar \eta_{k(j-1)}$, we get $\frac 1 {\sqrt p} \eta_{kj}$ and $-\frac 1 {\sqrt p} \bar \eta_{k(j-1)}$, which are the roots in \eqref{odd}. This proves the theorem.
\end{proof}

We write $\mathrm{mult}(a_k, \alpha)$ for the multiplicity of $\alpha$ in $\mathcal H(a_k)$ for $k \in \mathbb N$.
\begin{Cor} \label{lower}
Assume that  we have either $\psi_l(\beta) =\eta_{kj}\psi_k(\alpha)$ and $l=2kj$ for some $j\in \mathbb N$, or $\psi_l(\beta) =\eta_{k(j-1)}\psi_k(\alpha)$ and $l=k(2j-1)$ for some $j \in \mathbb N$. Then we have the inequalities:
\[ \mathrm{mult}(a_l, \beta) \le \mathrm{mult}(a_k, \alpha).\]
\end{Cor}

\begin{Rmk}
An upper bound for $\mathrm{mult}(a_k, \alpha)$ is given by the homogeneous dimension of the corresponding free Lie algebra (see \cite{KaMe}). Since the depth (or height) of $\beta$ is much smaller than that of $\alpha$, the number $\mathrm{mult}(a_l, \beta)$ can be considered as a lower bound for $\mathrm{mult}(a_k, \alpha)$.
\end{Rmk}

\subsection{Automorphic correction of $\mathcal H(a)$} \label{H-correct}

In the rest of this section, we will construct automorphic correction of  $\mathcal H(a)$ for $a=3, 11, 66$. Hence, $a=a_1$ for each prime $p \in \{ 5, 13, 17 \}$ and we will write $\psi=\psi_1$ for convenience. Recall that we need to establish data (1)-(4) (Section \ref{correction}). We already have data (1)-(3). More precisely, we put \[ M=\psi^{-1}(\frak d^{-1}) \subset \frak h^*_{\mathbb Q} \quad \text{ for each }\Bbb Q(\sqrt p), \ p=5, 13, 17 ,\] and use the same bilinear form on $\frak h^*_{\Bbb Q}$. Then $M$ is of signature $(1,1)$. We take the same Weyl group $W$ for the reflection group of $M$, and choose the cone
\begin{equation} \label{eqn-V} V^+(M) = \{ x \gamma^+ + y \gamma^- \in \mathfrak h^*_{\mathbb R} \, |\,  x >0 , \ y>0 \}. \end{equation} We set $\Pi = \{ \alpha_1, \alpha_2 \}$ and obtain the Weyl chamber \[ \mathcal M =\{ \beta \in V^+(M) \, | \, (\beta, \alpha_i) \le 0 , \ i=1,2 \} = \mathbb R_{\le 0}\, \omega_1+ \mathbb R_{\le 0}\, \omega_2 . \] The Weyl vector is given by  $\rho = -(\omega_1 +\omega_2)$. The Cartan matrix is the same $A$ for $\mathcal H(a)$.

Now we consider the data (4). We have the complexified cone \[ \Omega(V^+(M))= M \otimes \mathbb R + i V^+(M) = \left \{ \binom {z_1} {z_2} : Im(z_1)>0, \ Im(z_2)>0 \right \} \subset \mathfrak h^*\] with respect to the basis $\{ \gamma^+ , \gamma^- \}$ and from our choice of $V^+(M)$ in \eqref{eqn-V}. Then $\Omega(V^+(M))$  is naturally identified with $\mathbb H^2$. We choose the extended lattice $L= \begin{pmatrix} 0 & -p \\ -p & 0 \end{pmatrix} \oplus M$, which is essentially identical to $L$ in \eqref{eqn-latt}. Then the space $\mathcal K^+$ is given by
\begin{eqnarray*} \mathcal K^+ &=& \left \{ [ \tfrac {(z,z)}{2p} e_1 + e_2 + z ] \in P(L(\mathbb C)) : z= \tbinom {z_1}{z_2} \in \Omega(V^+(M)) \right \} \\ & =& \left \{ [ -z_1z_2 e_1 + e_2 + \tbinom {z_1}{z_2} ] \in P(L(\mathbb C)): (z_1, z_2) \in \mathbb H^2 \right  \} \\ & \cong &  \left \{ \left  [ \begin{pmatrix} z_1z_2 & z_1 \\ z_2 & 1 \end{pmatrix} \right  ]\in P(M_2(\mathbb C)) : (z_1, z_2) \in \mathbb H^2 \right  \}. \end{eqnarray*} The last identification follows from \eqref{eqn-metry}. The action of $SL_2(\mathcal O)$ on $\mathbb H^2 \cong \Omega(V^+(M))$ is compatible with its action on $M_2(\mathbb C)$, and we have $SL_2(\mathcal O)= \Gamma_F \subset O^+_L$. As we observed in Section \ref{subsec-H}, an automorphic form on $\Omega(V^+(M))$ is a Hilbert modular form. Hence an automorphic correction of $\mathcal H (a)$ is a Hilbert modular form which can be written as a product. We obtain natural examples from Section \ref{Hilbert} where we considered the works of Bruinier and others on Hilbert modular forms as Borcherds products (\cite{Bor-98, BB}).

Actually, our automorphic correction will be a Hilbert modular form with respect to the congruence subgroup $\Gamma_0(p)$ defined by
\[ \Gamma_0(p) = \left \{ \begin{pmatrix} a & b \\ c & d \end{pmatrix} \in SL_2(\mathcal O) : a, b, d \in \mathcal O, \ c \in (p) \right \} \subset O^+_L ,\] where $(p) \subset \mathcal O$ is the principal ideal generated by $p$. The notation $\Gamma_0(p)$ is the same as the congruence subgroup $\Gamma_0(p)$ of $SL_2(\mathbb Z)$. However, it will be clear from the context which group we mean.

\medskip

We fix $p \in \{ 5, 13, 17 \}$ and consider $\mathcal H(a_1)$. Recall that $a_1=3$ for $p=5$, $a_1=11$ for $p=13$ and $a_1=66$ for $p=17$.
We identify $\mathbb H^2$ with $\Omega (V^+(M)) \subset \mathfrak h^*$ as above. Then the Weyl group $W$ acts on $\mathbb H^2$; in particular, we have \[ r_1(z_1, z_2) = (\eta_1^2 z_2, \bar \eta_1^2 z_1) \quad \text{ and } \quad r_2(z_1, z_2) =(z_2, z_1) . \] Recall that the Weyl group $W$ also acts on $F$ by \[ r_1 \nu= \eta_1^2 \bar \nu \quad \text{ and } \quad r_2 \nu = \bar \nu \qquad \text{ for } \nu \in F.\]
For $m \in \mathbb N$, we define a map $\psi^{(m)}: \mathfrak h^*_{\mathbb Q} \rightarrow F$ by $\binom \nu {\bar \nu} \mapsto m \nu$, i.e. $\psi^{(m)} =m\psi$. Then we obtain
\begin{equation} \label{eqn-rw}   \psi^{(m)} (\rho)= \frac m {sp} (1 +\eta_1) = \frac {m \varepsilon_0}{\mathrm{tr}(\varepsilon_0) \sqrt p} .\end{equation}

\begin{Lem} \label{lem-pos}
Assume that $m=q_1^{k_1}\cdots q_r^{k_r}$ is the prime factorization of $m$ with distinct primes $q_i$ and suppose that
$( \frac {q_i} p )\neq 1$ for all $i$. Then we have $\nu \in \psi^{(m)}(\Delta^+_{\mathrm{re}})$ if and only if $\nu \in \mathfrak d^{-1}$, $\nu+2 \bar \nu >0$ and $N(\nu)=-m^2/p$.
\end{Lem}

\begin{proof}
The ``only if" part can be verified through straightforward computations. Consider the ``if" part. Since no $q_i$ splits in $F$, we obtain $\nu=\pm \frac m {\sqrt p} \varepsilon_0^{2j}$ for some $j \in \mathbb Z$ from the conditions $N(\nu)=-m^2/p$ and $\nu \in \mathfrak d^{-1}$. Recall that $\varepsilon_0^2=\eta_1$. We obtain from the description of the positive real roots \eqref{eqn-re} that
the additional condition $\nu + 2 \bar \nu >0$ makes $\nu \in \psi^{(m)}(\Delta^+_{\mathrm{re}})$.
\end{proof}

\begin{Rmk} \label{rmk-re}
It is important to notice that the same conditions as in the above lemma appear  in the Borcherds lift $\Psi_{m^2}$ in \eqref{eqn-Bor-2}. In particular, we can write for such an $m$
\[ \prod_{\substack{\nu\in\frak{d}^{-1} ,\ \nu+ 2\bar\nu>0 \\ N(\nu)=-m^2/p} }
(1-e(\nu z_1+\bar \nu z_2)) =\prod_{\nu\in \psi^{(m)}(\Delta^+_{\mathrm{re}}) }
(1-e(\nu z_1+\bar \nu z_2)) .\]
\end{Rmk}

\begin{Prop} \label{prop-det}
Let $\Psi_m$ be the Borcherds lifts for $m \in \mathbb N$. Define $\overline{\Psi}_m(z_1, z_2) = \Psi_m(z_2, z_1)$ and write $m=q_1^{k_1}\cdots q_r^{k_r}$ into the prime factorization with distinct primes $q_i$.
   \begin{enumerate}

  \item  Assume that $( \frac {q_i} p )=-1$ with an odd $k_i$ for some $i$. Then we have
  \[ \overline \Psi_m(wz)= \overline \Psi_m(z) \qquad \text{ for } w \in W.\]

   \item Assume that
$( \frac {q_i} p )\neq 1$ for all $i$. Then we have  \[ \overline \Psi_{m^2}(wz)= \det(w)\overline \Psi_{m^2}(z) \qquad \text{ for } w \in W.\]

\end{enumerate}
\end{Prop}

\begin{proof}
We have only to consider the simple reflections $r_1$ and $r_2$. First, consider the part (1). In this case, the Borcherds product $\Psi_m$ is of the form \eqref{eqn-Bor-1}. It is easy to see that \[ \nu \gg 0 \quad \Longleftrightarrow \quad \bar \nu \gg 0 \quad \Longleftrightarrow \quad \bar \nu \eta_1^2 \gg 0 .\] Then we obtain $\overline \Psi_m(r_1z) = \overline \Psi_m(r_2z) = \overline \Psi_m(z)$.

Now we consider the part (2). From \eqref{eqn-Bor-2} and Remark \ref{rmk-re}, we have
\begin{eqnarray*}
&{}& \overline \Psi_{m^2}(z_1,z_2) = \Psi_{m^2}(z_2, z_1) \\
&=& e\left( \frac {m\varepsilon_0 z_2}{\mathrm{tr}(\varepsilon_0)\sqrt{p}}-\frac {m \bar \varepsilon_0 z_1}{\mathrm{tr}(\varepsilon_0)\sqrt{p}}\right)\prod_{\substack{\nu\in\frak{d}^{-1} \\ \nu \gg 0} }
(1-e(\nu z_2+\bar \nu z_1))^{s(p\nu\bar \nu)a(p\nu\bar \nu)} \prod_{\nu\in \psi^{(m)}(\Delta^+_{\mathrm{re}})  }
(1-e(\nu z_2+\bar \nu z_1)). 
\end{eqnarray*}
The product over $\nu \gg 0$ is invariant under $r_1$ and $r_2$ as in the part (1). Write $\nu_i= \psi^{(m)}(\alpha_i)$, $i=1,2$. Each $r_i$ sends $\nu_i$ to $-\nu_i$ and keeps the set $\psi^{(m)}(\Delta^+_{\mathrm{re}} \setminus \{ \alpha_i \} )$ invariant. For $w=r_1$, we have
\begin{eqnarray*}
& & \overline \Psi_{m^2}(r_1(z_1,z_2)) = \overline \Psi_{m^2}(\eta_1^2 z_2, \bar \eta_1^2 z_1) =  \Psi_{m^2}(\bar \eta_1^2 z_1, \eta_1^2 z_2) \\
&=& A_1 \prod_{\substack{\nu\in\frak{d}^{-1} \\ \nu \gg 0} }(1-e(\nu z_1+\bar \nu z_2))^{s(p\nu\bar \nu)a(p\nu\bar \nu)} \prod_{\nu\in \psi^{(m)}(\Delta^+_{\mathrm{re}})  }
(1-e(\nu \bar \eta_1^2 z_1+\bar \nu \eta_1^2 z_2)) \\ & =& A_1 \prod_{\substack{\nu\in\frak{d}^{-1} \\ \nu \gg 0} }(1-e(\nu z_2+\bar \nu z_1))^{s(p\nu\bar \nu)a(p\nu\bar \nu)} \prod_{\nu\in \psi^{(m)}(\Delta^+_{\mathrm{re}})  }
(1-e(r_1 \nu z_2+\overline {r_1\nu} z_1)) \\
&=& A_1  \  \frac { 1-e(-\nu_1 z_2 - \bar \nu_1 z_1)}{1-e(\nu_1 z_2+\bar \nu_1 z_1)} \prod_{\substack{\nu\in\frak{d}^{-1} \\ \nu \gg 0} }(1-e(\nu z_2+\bar \nu z_1))^{s(p\nu\bar \nu)a(p\nu\bar \nu)} \prod_{\nu\in \psi^{(m)}(\Delta^+_{\mathrm{re}})  }
(1-e(\nu z_2+\bar \nu z_1))\\ &=& - A_1 \ e( -\nu_1 z_2 - \bar \nu_1 z_1)  \prod_{\substack{\nu\in\frak{d}^{-1} \\ \nu \gg 0} }
(1-e(\nu z_2+\bar \nu z_1))^{s(p\nu\bar \nu)a(p\nu\bar \nu)} \prod_{\nu\in \psi^{(m)}(\Delta^+_{\mathrm{re}})  }
(1-e(\nu z_2+\bar \nu z_1)),
\end{eqnarray*}
where we put \[ A_1 = e\left( \frac {m\varepsilon_0 \bar \eta_1^2 z_1}{\mathrm{tr}(\varepsilon_0)\sqrt{p}}-\frac {m \bar \varepsilon_0 \eta^2 z_2}{\mathrm{tr}(\varepsilon_0)\sqrt{p}}\right) .\] We obtain from \eqref{eqn-rw} that
\[ -\frac {m \bar \varepsilon_0 \eta^2}{\mathrm{tr}(\varepsilon_0)\sqrt{p}} - \nu_1 = \psi^{(m)}(r_1 \rho) - \psi^{(m)}(\alpha_1) = \psi^{(m)}(\rho) =  \frac {m\varepsilon_0 }{\mathrm{tr}(\varepsilon_0)\sqrt{p}} .\]
Combining these computations, we see that $\overline \Psi_{m^2}(r_1(z_1,z_2)) = -\overline \Psi_{m^2}(z_1, z_2)$. Similarly, we can show that $\overline \Psi_{m^2}(r_2(z_1,z_2)) = -\overline \Psi_{m^2}(z_1, z_2)$.
\end{proof}

\medskip

We define $\Phi_m(z)=\overline \Psi_m(pz)$. Then the function $\Phi_m(z)$ is a Hilbert modular form with respect to $\Gamma_0(p)$ thanks to the following lemma.
\begin{Lem} \label{lem-pz}
Let $g(z)$ be a Hilbert modular form for $\mathbb Q(\sqrt p)$ with respect to $SL_2(\mathcal O)$. Define $f(z)=\overline g(pz)$, where $\overline g(z_1, z_2) = g(z_2, z_1)$. Then the function $f(z)$ is a Hilbert modular form with respect to the congruence subgroup $\Gamma_0(p)$.
\end{Lem}

\begin{proof}
By Theorem 4.4.4 in \cite{M}, the function $\overline g(z)$ is a Hilbert modular form with respect to $SL_2(\mathcal O)$. Assume that $\mu$ is the multiplier system for $\overline g$. Then we define $\tilde \mu$ on $\Gamma_0(p)$ by
\[ \tilde \mu  \begin{pmatrix} a & b \\ pc & d \end{pmatrix}  =  \mu \begin{pmatrix} a & bp \\ c & d \end{pmatrix} , \qquad  \begin{pmatrix} a & b \\ pc & d \end{pmatrix} \in \Gamma_0(p) .\]
For $\gamma \in \begin{pmatrix} a & b \\ pc & d \end{pmatrix} \in \Gamma_0(p)$, we have
\begin{eqnarray*}
f(\gamma z) &=& \overline g(p\, \gamma z) = \overline g \left ( \frac {a pz+bp}{cpz+d} \right ) \\ &=& \mu \begin{pmatrix} a & bp \\ c & d \end{pmatrix}  N(cp z+d)^k \overline g(pz) = \tilde \mu \begin{pmatrix} a & b \\ pc & d \end{pmatrix}  N(pc z+d)^k f(z).
\end{eqnarray*}
Thus $f(z)$ is a Hilbert modular for with respect to $\Gamma_0(p)$ with the multiplier system $\tilde \mu$.
\end{proof}

Since  $F \cong \mathfrak h^*_{\mathbb Q} \subset \mathfrak h^*$, $\nu \mapsto \binom {\nu}{\bar \nu}$, and $\mathbb H^2 \cong \Omega(V^+(M)) \subset \mathfrak h^*$, $(z_1,z_2) \mapsto \binom {z_1}{z_2}$, the symmetric bilinear form on $\mathfrak h^*$ induces a paring on $F \times \mathbb H^2$ given by
\[ (\nu, z)= -p\,(\nu z_2+ \bar \nu z_1) \qquad \text{for } \nu \in F \text{ and } z=(z_1, z_2) \in \mathbb H^2. \]

Write $m=q_1^{k_1}\cdots q_r^{k_r}$ as before. If $( \frac {q_i} p )=-1$ with an odd $k_i$ for some $i$,then we can rewrite \eqref{eqn-Bor-1} and obtain
  \begin{equation} \label{eqn-bb} \Phi_m(z) = \prod_{\substack{\nu\in\frak{d}^{-1} \\ \nu\gg 0}} (1-e(p\,(\nu z_2+\bar \nu z_1)))^{s(p\nu\bar \nu)a(p\nu\bar \nu)} = \prod_{\substack{\nu\in\frak{d}^{-1} \\ \nu\gg 0}} (1-e(-(\nu,z))^{s(p\nu\bar \nu)a(p\nu\bar \nu)}.\end{equation}
We write $\rho_m=\psi^{(m)}(\rho)$. If $( \frac {q_i} p )\neq 1$ for all $i$, then we obtain from \eqref{eqn-Bor-2}, \eqref{eqn-rw} and Remark \ref{rmk-re},

\begin{eqnarray}
\Phi_{m^2}(z) &=&  e\left( p\, (\rho_m z_2 + \bar \rho_m z_1)\right)\prod_{\substack{\nu\in\frak{d}^{-1} \\ \nu \gg 0} }
(1-e(p\, (\nu z_2+\bar \nu z_1))^{s(p\nu\bar \nu)a(p\nu\bar \nu)} \nonumber \\ & & \phantom{LLLLLLLLLLL} \times \prod_{\nu\in \psi^{(m)}(\Delta^+_{\mathrm{re}}) }
(1-e(p\, (\nu z_2+\bar \nu z_1)) \nonumber \\ &=&  e\left( -(\rho_m,z) \right)\prod_{\substack{\nu\in\frak{d}^{-1} \\ \nu \gg 0} }
(1-e(-(\nu,z))^{s(p\nu\bar \nu)a(p\nu\bar \nu)} \prod_{\nu\in \psi^{(m)}(\Delta^+_{\mathrm{re}})}
(1-e(-(\nu,z)) .  \label{eqn-bq}
\end{eqnarray}

We write $f_1=q^{-1} + \sum_{n\ge0} a(n) q^n$. When $p=17$, we assume that $s(n)a(n)$ are integers for all $n \ge 1$. This is necessary since $s(n)a(n)$ will be considered as root multiplicities in what follows. See the remark at the end of Section \ref{weakly}.  Now we state the main theorem of this paper.

\begin{Thm} \label{main}
Let $p \in \{5, 13,17 \}$.  Then the Hilbert modular form $\Phi_1$
provides an automorphic correction for the hyperbolic Kac-Moody algebra $\mathcal H(a_1)$, where $a_1=3$ for $p=5$, $a_1=11$ for $p=13$ and $a_1=66$ for $p=17$. In particular, there exists a generalized Kac-Moody superalgebra $\widetilde{\mathcal H}$ whose denominator function is the Hilbert modular form $\Phi_1$.
\end{Thm}

\begin{proof}
From \eqref{eqn-bq}, we have \[
\Phi_{1}(z) =  e\left(-(\rho,z) \right)\prod_{\substack{\nu\in\frak{d}^{-1} \\ \nu \gg 0} }
(1-e(-(\nu,z))^{s(p\nu\bar \nu)a(p\nu\bar \nu)} \prod_{\nu\in \psi(\Delta^+_{\mathrm{re}})}
(1-e(-(\nu,z)) .
\]
We will drop $\psi$ from the notation. We write $\Phi_1(z) = \sum_{\mu} b_\mu e(-(\mu, z))$. By Proposition \ref{prop-det}, we obtain $\Phi_1(wz) = \det (w) \Phi_1(z)$ for $w \in W$. Since $\Phi_1(wz) = \sum_{\mu} b_{w\mu}e(-(\mu,z))$, we get \begin{equation} \label{anti} b_{w\mu}= \det (w) b_{\mu}.\end{equation} One can easily see that $\rho+\nu \in V^+(M)$ for $\nu \in \frak d^{-1}$, $\nu \gg 0 $ and for $\nu \in \Delta^+_{\mathrm{re}}$. Hence the sum is over $\mu \in V^+(M)$ such that $\mu -\rho \in M$. Then we can write, using the fundamental chamber $\mathcal M \subset V^+(M)$,
\[ \Phi_1(z) = \sum_{w \in W} \det(w) \left ( - \sum_{\substack{\rho+ \nu \in \mathcal M \\ \nu \in M} } m(\nu) e(-(w(\rho+\nu), z)) \right ).\]

Assume that $\rho+\nu \in \mathcal M$. If $(\rho+\nu, \alpha_i)=0$ for $i=1, 2$, then $\rho+\nu$ is invariant under $r_i$, and $m(\nu)=0$ from \eqref{anti}. Thus we may assume $(\rho+\nu, \alpha_i) <0$. Then we have $(\nu, \alpha_i) \le 0$ for $i=1, 2$ and $\nu \in \mathcal M$ if $\nu \neq 0$. Since $m(0)=-1$, we have
\[ \Phi_1(z) = \sum_{w \in W} \det(w) \left ( e(-(w(\rho), z)) - \sum_{\nu \in M \cap \mathcal M } m(\nu) e(-(w(\rho+\nu), z)) \right ).\]
This is exactly of the form required by the item (4) for an automorphic correction in Section \ref{correction}. The data (1)-(3) have already been established in the beginning of Section \ref{H-correct}.
The existence of the corresponding generalized Kac-Moody superalgebra $\widetilde{\mathcal H}$ is a consequence of the theory of an automorphic correction as explained in Section \ref{correction}.
\end{proof}

\begin{Rmk}
There are some obstructions when we try to interpret $\Phi_m$ $(m \neq 1)$ as an automorphic correction. When we have $\Phi_m$ of the form \eqref{eqn-bb}, we do not have the product corresponding to the real roots. When the function $\Phi_{m^2}$ is of the form \eqref{eqn-bq}, we have both parts corresponding to the real roots and to the imaginary roots. However, the map $\psi^{(m)}$ is not an isometry. If we make $\psi^{(m)}$ an isometry by adjusting the bilinear form on $F$, the natural lattice would be $m \frak d^{-1}$. Then $\Phi_{m^2}$ is not an automorphic correction of $\mathcal H (a_1)$.
\end{Rmk}

\vskip 1 cm

\section{Asymptotics for Root Multiplicities}

In this section, we obtain asymptotics of Fourier coefficients of the modular forms $f_m$ defined in Section \ref{weakly}. Note that the Fourier coefficients of $f_1$ are root multiplicities of the generalized Kac-Moody superalgebra $\widetilde{\mathcal H}$ with some modification.

We recall the result of J. Lehner \cite{L1} on Fourier coefficients of modular forms using the method of Hardy-Ramanujan-Rademacher.
We refer to \cite{L1} for unexplained notations:
Let $f(z)$ be a weakly homomorphic modular form of weight $0$ with respect to $\Gamma$. Let $p_0=\infty, p_1,...,p_{s-1}$ be the cusps of $\Gamma$,
and
$$A_0=\begin{pmatrix} 1&0\\0&1\end{pmatrix},\quad A_j=\begin{pmatrix} 0&-1\\1&-p_j\end{pmatrix},\, j>0.
$$
Let $M^*=A_j M=\begin{pmatrix} a&b\\c&d\end{pmatrix}$ for $M\in \Gamma$, and let
\begin{eqnarray*} C_{j0} &=& \{ c\, |\, \begin{pmatrix} \cdot &\cdot\\ c&\cdot\end{pmatrix}\in A_j\Gamma \}, \\
                  D_c    &=& \{ d\, |\, \begin{pmatrix} \cdot &\cdot\\ c&d\end{pmatrix}\in A_j\Gamma,\, 0<d\leq c \}.
\end{eqnarray*}
It can be shown (\cite{L1}, page 313) that given such $c,d$, there is a unique $a$ such that $-c\lambda_j\leq a<0$. For $k=1,...,s-1$, let
$$e \left ( - \tfrac {\kappa_k} {\lambda_k} {A_k z} \right ) f(z)=\sum_{n=-\mu_k}^\infty a(n)^{(k)} q_k^n,\quad q_k=e \left ( \tfrac {A_k z}{\lambda_k} \right ),
$$
where $\kappa_k, \lambda_k$ are defined as in \cite{L2}, page 398.
By replacing $A_k z$ by $z$, this can be written as
$$f(p_k-\tfrac 1z)= q^{\frac {\kappa_k}{\lambda_k}} \sum_{n=-\mu_k}^\infty a(n)^{(k)} q^{\frac n{\lambda_k}}.
$$
For $k=0$, we have the usual Fourier expansion: (We assume that $\lambda_0=1, \kappa_0=0$ for $\Gamma$.)
$$f(z)=\sum_{n=-\mu_0}^\infty a(n) q^n.
$$

\begin{Thm} \cite{L1} For $n>0$,
\begin{equation} \label{fourier} a(n)=2\pi  \sum_{j=0}^{s-1} \sum_{\nu=1}^{\mu_j} a(-\nu)^{(j)} \sum_{c\in C_{j0}} c^{-1}A(c,n,\nu_j)M(c,n,\nu_j,0),
\end{equation}
where $\nu_j=\frac {\nu-\kappa_j}{\lambda_j}$, and
\begin{eqnarray*}
A(c,n,\nu_j) &=& \sum_{d\in D_c} v^{-1}(M) e \left ( \frac {nd-\nu_ja}c \right ), \quad M=A_j^{-1} M^*\\
M(c,n,\nu_j,0) &=& \left (\frac {\nu_j}n \right )^{\frac 12} I_1 \left (\frac {4\pi\sqrt{n\nu_j}}c \right ).
\end{eqnarray*}
\end{Thm}
In \cite{L1}, page 314, the formula for $a(n)$ has an error term. However, by imitating \cite{R2}, one can get a convergent series. (cf. \cite{D})

Now we apply the theorem to $f_m\in A_0^+(p,\chi_p)$ for $\chi_p(m)\ne -1$:
$$f_m=s(m)^{-1}q^{-m}+\sum_{n=0}^\infty a_m(n) q^n.
$$
Here $\Gamma_0(p)$ has two cusps; $p_0=\infty,\, p_1=0$. (\cite{Ko}, page 108)

First, $p_0=\infty$. In this case, $\lambda_0=1, \kappa_0=0$. If $c\in C_{00}$, then $p|c$, and the smallest $c\in C_{00}$ is $p$.
So
$$A(p,n,m)=\sum_{d\in D_p} v^{-1}(M)  e\left(\frac {nd-ma}p\right)=
\sum_{d=1}^{p-1} \chi_p(d)  e\left(\frac {nd-ma}p \right),
$$
where $ad\equiv 1$ (mod $p$). This is
 the Sali\'e sum $T(n,-m; p)$ (\cite{IK}, page 323):
If $p\nmid m$, then
$$
A(p,n,m)=T(n,-m; p)=\sqrt{p} \left(\frac {-m}p\right) \sum_{v^2\equiv -mn\, (\mathrm{mod} \,  p)} e\left(\frac {2v}p \right).
$$
Since $\chi_p(m)=1$, we have $A(p,n,m)=\sqrt{p} \displaystyle{\sum_{v^2\equiv -mn\, (\mathrm{mod} \, p)}} e\left(\frac {2v}p \right)$.
Note that if $p|n$, then we obtain $A(p,n,m)=\sqrt{p}$.
If $p|m$,
$$A(p,n,m)=\sum_{d\in D_p} \chi_p(d) e\left(\frac {nd}p \right)=
\begin{cases}  \chi_p(n) \sqrt{p}, &\text{if $p\nmid n$;}\\
0, &\text{if $p|n$}.\end{cases}
$$

Second, $p_1=0$. In this case, $\lambda_1=p, \kappa_1=0$, and $A_1=\begin{pmatrix} 0&-1\\1&0\end{pmatrix}$. The smallest $c\in C_{10}$ is 1. Thus we get $A(1,n,\nu_1)=1$.
In order to compute the Fourier expansion of $f_m$ at 0, we use \cite{BB}, page 54:
$$f_m \left (-\frac 1{pz} \right )=f_m|W_p(z)=\frac 1{\sqrt{p}} f|U_p (z)=\sqrt{p} \sum_{n\in\Bbb Z} a_m(pn) q^n.
$$
Hence
$$f_m \left (-\frac 1{z} \right )=\sqrt{p} \sum_{n\in\Bbb Z} a_m(pn) q^{\frac np}
=\begin{cases} \sqrt{p} \displaystyle\sum_{n=0}^\infty a_m(pn) q^{\frac np}, &\text{if $p\nmid m$};\\
\frac {\sqrt{p}}{s(m)} q^{-\frac {m}{p^2}}+\sqrt{p}\displaystyle \sum_{n=0}^\infty a_m(pn) q^{\frac np}, &\text{if $p|m$} .
\end{cases}
$$
So by (\ref{fourier}), we have, if $p\nmid m$,
\begin{eqnarray*}
a_m(n) &=& 2\pi \sum_{c\in C_{00}} \frac {A(c,n,m)}c \left (\frac mn \right )^{\frac 12} I_1 \left (\frac {4\pi\sqrt{nm}}c \right ) \\
&=& \frac {2\pi T(n,-m; p)}p \left (\frac mn \right )^{\frac 12} I_1 \left (\frac {4\pi\sqrt{nm}}p \right )+\text{error term}.
\end{eqnarray*}
If $p|m$,
\begin{eqnarray*}
a_m(n) &=& \pi \sum_{c\in C_{00}} \frac {A(c,n,m)}c \left (\frac mn \right )^{\frac 12} I_1 \left (\frac {4\pi\sqrt{nm}}c \right ) +\frac {\pi}{\sqrt{p}}\sum_{c\in C_{10}}  \frac {A(c,n,\frac m{p^2})}c
\left (\frac mn \right )^{\frac 12} I_1 \left (\frac {4\pi\sqrt{nm}}{pc} \right ) \\
&=& \frac {\pi}{\sqrt{p}} \left (\frac mn \right )^{\frac 12} I_1 \left (\frac {4\pi\sqrt{nm}}p \right )(\chi_p(n)+1)+\text{error term}.
\end{eqnarray*}
Note that by definition, if $\chi_p(n)=-1$, $a_m(n)=0$.

Now we show that the error term is smaller than the main term. In the case of $p|m$, the second term is similar to the first term. So it is enough to handle the case $p\nmid m$.
By Weil's bound,
$$|A(c,n,m)|\leq (c,n,m)^{\frac 12} c^{\frac 12}\tau(c)\leq (n,m)^{\frac 12} c^{\frac 12}\tau(c)\leq (mn)^{\frac 12}c^{\frac 12}\tau(c),
$$
where $\tau(c)$ is the number of positive divisors of $c$.
We divide the error term into two regions: $p<c\leq 4\pi\sqrt{mn}$ and $c> 4\pi\sqrt{mn}$.
Here
\begin{eqnarray*}
2\pi \sum_{p<c\leq 4\pi\sqrt{mn}} \frac {A(c,n,m)}c \left (\frac mn \right )^{\frac 12} I_1 \left (\frac {4\pi\sqrt{nm}}c \right )
\leq
2\pi m I_1 \left (\frac {2\pi\sqrt{nm}}p \right ) \sum_{p<c\leq 4\pi\sqrt{mn} } \frac {\tau(c)}{\sqrt{c}} \\
\leq 8 \pi^{\frac 32} m^{\frac 54}n^{\frac 14} (\log 4\pi\sqrt{mn} )\, I_1 \left (\frac {2\pi\sqrt{nm}}p \right ).
\end{eqnarray*}
Here we used the fact that $\sum_{c\leq x} \frac {\tau(c)}{\sqrt{c}}\leq 2\sqrt{x}\log x$.

On the other hand, since $I_1(z)\leq z$ for $0<z<1$,
\begin{eqnarray*}
2\pi \sum_{c> 4\pi\sqrt{mn}} \frac {A(c,n,m)}c  \left (\frac mn \right )^{\frac 12} \, I_1 \left (\frac {4\pi\sqrt{nm}}c \right )
\leq
8\pi^2 m^{\frac 32}n^{\frac 12} \sum_{c> 4\pi\sqrt{mn}}  \frac {\tau(c)}{c^{\frac 32}}
\leq 48 \pi^{\frac 32} m^{\frac 54}n^{\frac 14} (\log 4\pi\sqrt{mn}).
\end{eqnarray*}
Here we used the fact that $\sum_{c> x} \frac {\tau(c)}{c^{\frac 32}}\leq 12x^{-\frac 12}\log x$.

Combining the above computations, we have proved the following theorem:
\begin{Thm} \label{last} For a positive integer $m$ with $\chi_p(m)\ne -1$, let $f_m=s(m)^{-1}q^{-m}+\sum_{n=0}^\infty a_m(n)q^n\in A_0^+(p,\chi_p)$. Then for any $m$, $a_m(n)>0$ for all $n$, $p|n$.

If $p|m$, we have $a_m(n)\geq 0$ for all $n$, and 
$$a_m(n)=\frac {\pi}{\sqrt{p}} \left (\frac mn \right )^{\frac 12} \, I_1 \left (\frac {4\pi\sqrt{mn}}p \right ) (\chi_p(n)+1) + O\left(m^{\frac 54}n^{\frac 14} \log 4\pi\sqrt{mn} \, I_1 \left (\frac {2\pi\sqrt{nm}}p \right )\right).
$$
If $p\nmid m$, we obtain
$$a_m(n)=\frac {2\pi}{\sqrt{p}} \left (\frac mn \right )^{\frac 12} \, I_1 \left (\frac {4\pi\sqrt{mn}}p \right ) \left(\sum_{v^2\equiv -mn \, (\mathrm{mod} \, p)} e \left (\frac {2v}p \right )\right) + O\left(m^{\frac 54}n^{\frac 14} (\log 4\pi\sqrt{mn}) \, I_1 \left (\frac {2\pi\sqrt{nm}}p \right ) \right).
$$
\end{Thm}

For example, let $p=5$, $m=6$, $n=9$. In this case, $\sum_{v^2\equiv -mn \,(\mathrm{mod}\, p)} e(\frac {2v}5)=2\cos \frac {4\pi}5$. So
$a_6(9)\sim \frac {2\pi}{\sqrt{5}} (6/9)^{\frac 12} I_1 (\frac {4\pi\sqrt{54}}5) 2 \cos \frac {4\pi}5=-35409600$. The exact value is $-35408776$.
Let $p=5$, $m=10$, $n=9$. In this case, $a_{10}(9)\sim \frac {2\pi}{\sqrt{5}} (10/9)^{1/2} I_1(\frac {4\pi\sqrt{90}}5)=5391530000$. The exact value is
5391558200.

\end{document}